\documentclass[12pt,twoside]{amsart}
\usepackage{amssymb}
\usepackage{amscd}
\usepackage{amsmath}
\usepackage[all]{xy}

\title[Semi-stable MMP]{Semi-stable minimal model 
program for varieties with 
trivial canonical 
divisor}
\author{Osamu Fujino}
\date{2010/12/14, version 1.25}
\address{Department of Mathematics, Faculty of Science, 
Kyoto University, Kyoto 606-8502, Japan}
\email{fujino@math.kyoto-u.ac.jp}

\newcommand{\Supp}[0]{{\operatorname{Supp}}}
\newcommand{\Mov}[0]{{\operatorname{Mov}}}
\newcommand{\codim}[0]{{\operatorname{codim}}}
\newcommand{\Exc}[0]{{\operatorname{Exc}}}
\newtheorem{thm}{Theorem}[section]
\newtheorem{lem}[thm]{Lemma}
\newtheorem{cor}[thm]{Corollary}

\theoremstyle{definition}

\newtheorem{defn}[thm]{Definition}
\newtheorem{rem}[thm]{Remark}
\newtheorem*{ack}{Acknowledgments}       
\newtheorem*{notation}{Notation}         
\newtheorem{say}[thm]{}
\newtheorem{step}{Step}         
\begin{document}
\bibliographystyle{amsalpha+}

\maketitle 
\begin{abstract} 
We give a sufficient condition for the termination of 
flips. 
Then we discuss a semi-stable minimal model program 
for varieties with (numerically) trivial canonical divisor as 
an application. 
We also treat a slight refinement of dlt blow-ups.  
\end{abstract}

\tableofcontents

\section{Introduction}\label{sec1}

In this paper, we give a sufficient condition for the 
termination of flips. 
For the precise statement, see Theorem \ref{33}. 
By using this criterion:~Theorem \ref{33}, 
we prove the following theorem, which is 
a semi-stable minimal model program for 
varieties with trivial canonical divisor. 
It was inspired by Yoshinori Gongyo's 
paper \cite{gongyo} and 
Daisuke Matsushita's seminar talk on May 21, 2010 in Kyoto.  

\begin{thm}[Semi-stable minimal model program for varieties with 
trivial canonical divisor]\label{13}
Let $f:X\to Y$ be a proper surjective morphism from a smooth 
quasi-projective 
variety $X$ to a smooth 
quasi-projective curve $Y$ with connected fibers. 
Let $P\in Y$ be a point. 
Assume that 
$f^*P$ is a reduced simple normal crossing 
divisor on $X$ and $f$ is smooth 
over $Y\setminus P$. 
We further assume that $K_{f^{-1}Q}\sim 0$ 
for every $Q\in Y\setminus P$. 
Then there exists a sequence of flips and 
divisorial contractions 
$$
X=X_0\dashrightarrow X_1\dashrightarrow \cdots 
\dashrightarrow X_k\dashrightarrow 
\cdots\dashrightarrow X_m
$$ 
over $Y$ 
such that $K_{X_m}\sim _Y0$. 
We note that $X_m$ has only $\mathbb Q$-factorial 
terminal singularities. 
Moreover, the special fiber $S=f^{-1}_m P=f^*_m P$ 
of $f_m: X_m\to Y$ is Gorenstein, 
semi divisorial log terminal, and $K_S\sim 0$. 
\end{thm}

For the definition of {\em{semi divisorial log terminal}}, 
see \cite[Definition 1.1]
{fujino-semi}. 
For the proof of the 
termination of $4$-dimensional semi-stable log flips, see \cite{fujino1}. 
Theorem \ref{13} can be applied to semi-stable degenerations 
of Abelian varieties, 
Calabi-Yau varieties, and so on. From the minimal model theoretic viewpoint, 
the following theorem is a natural formulation of 
uniruled degenerations of varieties with numerically trivial 
canonical divisor (cf.~\cite[Theorem 1.1]{takayama}). 

\begin{thm}[Semi-stable minimal model program for 
varieties with numerically trivial canonical divisor]\label{14}
Let $f:X\to Y$ be a proper surjective morphism from a smooth 
quasi-projective 
variety $X$ to a smooth 
quasi-projective curve $Y$ with connected fibers. 
Let $P\in Y$ be a point. 
Assume that 
$f^*P$ is a reduced simple normal crossing divisor on $X$ and $f$ is smooth 
over $Y\setminus P$. 
We further assume that $K_{f^{-1}Q}\equiv 0$, 
equivalently, $K_{f^{-1}Q}\sim _{\mathbb Q}0$, 
for every $Q\in Y\setminus P$. 
Then there exists a sequence of flips and divisorial contractions 
$$
X=X_0\dashrightarrow X_1\dashrightarrow 
\cdots \dashrightarrow X_k\dashrightarrow 
\cdots\dashrightarrow X_m
$$ 
over $Y$ 
such that $K_{X_m}\sim _{\mathbb Q,Y}0$. 
We note that $X_m$ has only $\mathbb Q$-factorial 
terminal singularities. 
Moreover, the special fiber $S=f^{-1}_m P=f^*_m P$ of $f_m: X_m\to Y$ is 
semi divisorial log terminal and $K_S\sim_{\mathbb Q} 0$. 
Therefore, if $S$ is reducible, then 
every irreducible component of $S$ is uniruled. 
If $S$ is irreducible, then 
$S$ is uniruled if and only if $S$ is not canonical. 
\end{thm}

In this paper, we prove Theorem \ref{13} and Theorem \ref{14} as applications of 
the following theorem. 

\begin{thm}\label{main}
Let $(X, \Delta)$ be a $\mathbb Q$-factorial 
quasi-projective 
divisorial log terminal 
pair and let 
$f:X\to Y$ be a proper surjective morphism onto a smooth 
quasi-projective 
curve $Y$ with 
connected fibers. 
Assume that $(K_X+\Delta)|_F\sim _{\mathbb Q}0$ for a general 
fiber $F$ of $f$. 
Then there exists a sequence of flips and divisorial 
contractions 
\begin{align*}
(X, \Delta)=(X_0, \Delta_0)\dashrightarrow (X_1, 
\Delta_1)\dashrightarrow \cdots\\ \dashrightarrow 
(X_k, \Delta_k) 
\dashrightarrow \cdots \dashrightarrow (X_m, \Delta_m)
\end{align*}
over $Y$ such 
that 
$
K_{X_m}+\Delta_m
\sim _{\mathbb Q, Y}0$ 
where $\Delta_k$ is the pushforward of $\Delta$ on $X_k$ for every $k$. 
\end{thm}

\begin{rem}
It is known that $(K_X+\Delta)|_F\sim _{\mathbb Q}0$ if 
and only if $(K_X+\Delta)|_F\equiv 0$. See, 
for example, \cite[Theorem 1]{ckp} and \cite[Theorem 1.2]{gongyo}. 
\end{rem}

We can also prove the following theorem as an application of 
Theorem \ref{main}. We recommend the reader to compare it with 
Kodaira's classification of elliptic 
fibrations (cf.~\cite[V.~Examples]{BPV}). 

\begin{thm}[{{cf.~\cite[Theorem 1.1]{takayama}}}]\label{15}
Let $f:X\to Y$ be a proper surjective morphism from a smooth 
quasi-projective 
variety $X$ to a smooth 
quasi-projective curve $Y$ with connected fibers. 
Let $P\in Y$ be a point. 
Assume that 
$\Supp f^*P$ is a 
simple normal crossing divisor on $X$ and $f$ is smooth 
over $Y\setminus P$. 
We further assume that $K_{f^{-1}Q}\equiv 0$, 
equivalently, $K_{f^{-1}Q}\sim _{\mathbb Q}0$, 
for every $Q\in Y\setminus P$. 
Then there exists a sequence of flips and divisorial contractions 
$$
X=X_0\dashrightarrow X_1\dashrightarrow \cdots 
\dashrightarrow X_k\dashrightarrow 
\cdots\dashrightarrow X_m
$$ 
over $Y$ 
such that $X_m$ has only $\mathbb Q$-factorial terminal 
singularities and 
$K_{X_m}\sim _{\mathbb Q,Y}0$. 
Let $S=\Supp f^*_mP$ be the special 
fiber of $f_m:X_m\to Y$. 
If $S$ is reducible, 
then every irreducible component of $S$ is uniruled. 
If $S$ is irreducible, then 
$S$ is normal and has only canonical 
singularities if and only if 
$S$ is not uniruled. 
We note that $K_S\sim _{\mathbb Q}0$ when 
$S$  is irreducible 
and has only canonical singularities. 
\end{thm}

By combining Theorem \ref{main} with \cite[Proposition 2.7]{lai}, 
we obtain the following result. 

\begin{cor}\label{12}
Let $f:X\to Y$ be a projective surjective morphism from a smooth 
quasi-projective 
variety $X$ onto a smooth 
quasi-projective 
curve $Y$ with connected fibers. 
Assume that 
the general fiber $F$ of $f$ has a good minimal 
model 
and $\kappa (F)=0$, where 
$\kappa (F)$ is the Kodaira dimension of $F$. 
Then there exists a sequence of flips and 
divisorial contractions 
$$
X=X_0\dashrightarrow X_1\dashrightarrow \cdots \dashrightarrow 
X_k\dashrightarrow \cdots\dashrightarrow X_m
$$ 
over $Y$ 
such that 
$K_{X_m}\sim _{\mathbb Q, Y}0$. 
\end{cor}

\begin{rem}
By \cite[Corollaire 3.4]{druel}, 
$F$ has a good minimal model with $\kappa (F)=0$ if and 
only if $\kappa _{\sigma}(F)=0$, 
where $\kappa _{\sigma}(F)$ is 
the numerical Kodaira dimension in the sense of Nakayama. 
See also \cite[Theorem 1.2]{gongyo}. 
\end{rem}

Finally, in Section \ref{sec-dlt}, we treat a slight 
refinement of {\em{dlt blow-ups}} (cf.~Theorem \ref{41}) 
as an application of our 
criterion for the termination of flips:~Theorem \ref{33}, which 
generalizes \cite[17.10 Theorem]{FA} and \cite[Corollary 1.4.3]{bchm}. 
We will use Theorem \ref{41} in the 
proofs of Theorem \ref{14} and Theorem \ref{15}. 

\begin{ack}
The author would like to thank Professor 
Takeshi Abe and Yoshinori Gongyo 
for useful discussions. 
He also likes to thank Professor 
Daisuke Matsushita for giving him various 
comments and answering his questions. 
He was partially supported by 
The Inamori Foundation and by the 
Grant-in-Aid for Young Scientists (A) $\sharp$20684001 
from JSPS. 
He thanks the referee whose comments and suggestions 
made this paper much better. 
\end{ack} 

\begin{notation}
Let $X$ be a normal variety and let $D=\sum _i a_i D_i$ 
be an $\mathbb R$-divisor on $X$, 
where 
$D_i$ is a prime divisor and $a_i\in \mathbb R$ for every $i$ and 
$D_i\ne D_j$ for 
every $i\ne j$. 
In this case, $D$ is called {\em{$\mathbb R$-boudnary}} 
if and only if $0\leq a_i\leq 1$ for every $i$. 

Let $f:X\to Y$ be a proper morphism of normal algebraic varieties. 
Two $\mathbb Q$-divisors $D_1$ and $D_2$ on $X$ are 
{\em{$\mathbb Q$-linearly equivalent over 
$Y$}}, 
denoted by $D_1\sim _{\mathbb Q, Y}D_2$, if their 
difference is a $\mathbb Q$-linear 
combination of principal divisors and a $\mathbb Q$-Cartier 
divisor pulled back from 
$Y$. 

Let $X$ be a normal variety and let $\Delta$ be 
an $\mathbb R$-divisor on $X$ such that 
$K_X+\Delta$ is $\mathbb R$-Cartier. 
Let $E$ be a divisor {\em{over}} $X$. Then 
the {\em{discrepancy}} coefficient of $E$ with respect to $(X, \Delta)$ is 
denoted by $a(E, X, \Delta)$. 
\end{notation}

We work over $\mathbb C$, the complex number field, 
throughout this paper. 
We freely use the standard terminology on the log minimal model program 
in \cite{bchm} and \cite{kollar-mori}. 

\section{Easy termination lemma}\label{sec-new}
In this section, we give a sufficient condition for 
the termination of flips. 
First, let us recall the definitions of {\em{movable divisors}} and 
the {\em{movable cone}}. 

\begin{defn}[Movable divisors and movable cone]
Let $f:X\to Y$ be a projective morphism 
of normal algebraic varieties. 
A Cartier divisor $D$ on $X$  is called 
{\em{$f$-movable}} if $f_*\mathcal O_X(D)\ne 0$ and if the cokernel of the natural 
homomorphism 
$f^*f_*\mathcal O_X(D)\to \mathcal O_X(D)$ has a support 
of codimension $\geq 2$. 

Let $M$ be an $\mathbb R$-Cartier $\mathbb R$-divisor 
on $X$. 
Then $M$ is called {\em{$f$-movable}} if and only if 
$M=\sum _i a_i D_i$ where $a_i$ is a positive real number 
and $D_i$ is an $f$-movable Cartier divisor for every $i$.  

We define 
$\overline {\Mov}(X/Y)$ as the closed convex cone in $N^1(X/Y)$, 
which is called the {\em{movable cone}} of 
$f:X\to Y$,  generated by the 
classes of $f$-movable Cartier divisors. 
\end{defn}

Let us recall the minimal model program with scaling 
(cf.~\cite[3.10]{bchm}, \cite[Definition 
3.2]{birkar}, and \cite[Theorem 18.9]{fujino}). 

\begin{say}[Minimal model program with scaling]\label{32}
Let $(X, \Delta)$ be a $\mathbb Q$-factorial 
dlt pair such that $\Delta$ is an $\mathbb R$-divisor and 
let $f:X\to Y$ be a projective surjective morphism between 
quasi-projective varieties. 
Let $H$ be an effective $\mathbb R$-divisor on $X$ such that 
$(X, \Delta+H)$ is divisorial log 
terminal, $K_X+\Delta+H$ is $f$-nef, and 
the relative augmented base 
locus $\mathbf B_+(H/Y)$ (cf.~\cite[Definition 3.5.1]{bchm}) 
contains no lc centers 
of $(X, \Delta)$.  
We run the $(K_X+\Delta)$-minimal model program with scaling of $H$ over $Y$. 
We obtain a sequence of 
divisorial contractions and 
flips 
$$
(X, \Delta)=(X_0, \Delta_0)\dashrightarrow (X_1, \Delta_1)\dashrightarrow \cdots 
\dashrightarrow (X_k, \Delta_k)\dashrightarrow \cdots 
$$ 
over $Y$. 
We note that 
$$
\lambda_i=\inf\{ t\in \mathbb R\, |\, K_{X_i}+\Delta_i +tH_i \ \text{is nef over $Y$}\}, 
$$ 
where $H_i$ (resp.~$\Delta_i$) is the 
pushforward of $H$ (resp.~$\Delta$) on $X_i$ for every $i$. 
By the definition, $0\leq \lambda_i\leq 1$ and $\lambda_i\in \mathbb R$ 
for every $i$ and 
$$\lambda_0\geq \lambda_1\geq \cdots \geq \lambda_k \geq \cdots.$$
We also note that the relative augmented base locus $\mathbf B_+(H_i/Y)$ contains 
no lc centers of $(X_i, \Delta_i)$ for every $i$ (cf.~\cite[Lemma 3.10.11]{bchm}). 
\end{say}

The following theorem is the main result of this section. 

\begin{thm}[Easy termination lemma]\label{33} 
Under the same notation as in {\em{\ref{32}}}, 
we assume that 
$H$ is big over $Y$, 
every step of the $(K_X+\Delta)$-minimal model program is 
a flip, and $K_X+\Delta\not\in \overline {\Mov}(X/Y)$. 
Then it terminates after finitely many steps. 
\end{thm}

\begin{proof}
We assume that the sequence does not terminate. 
First we assume that 
$$\lambda=\underset{i\to \infty}{\lim}\lambda_i>0.$$  
In this case, the sequence of flips we consider is a sequence of 
$(K_X+\Delta+\frac{1}{2}\lambda H)$-flips. 
We note that there exists an effective $\mathbb R$-divisor 
$B$ on $X$ such that 
$\Delta+\frac{1}{2}\lambda H\sim _{\mathbb R}B$,  
$(X, B)$ is klt,  $K_X+ B+(1-\frac{1}{2}\lambda)H$ 
is $f$-nef, $(X, B+(1-\frac{1}{2}\lambda)H)$ is 
klt, and 
$B$ is big over $Y$ (cf.~\cite[Lemma 3.7.3]{bchm} 
and \cite[Lemma 5.1]{gongyo}). 
Therefore there are no infinite sequences of 
flips by \cite[Corollary 1.4.2]{bchm}. 
It is a contradiction. 
Thus we can assume that $\lambda=0$. 
Under the assumption that $\lambda=0$, we will show that $K_X+\Delta\in \overline 
{\Mov} (X/Y)$. 
Let $G_i$ be a relative ample $\mathbb Q$-divisor on $X_i$ such 
that $G_{iX}\to 0$ in $N^1(X/Y)$ for $i\to \infty$ where 
$G_{iX}$ is the strict transform of $G_i$ on $X$. 
We note that $K_{X_i}+\Delta_i+\lambda_iH_i +G_i$ is ample over $Y$ for 
every $i$. 
Therefore the strict transform $K_{X}+\Delta+\lambda_iH +G_{iX}$ is 
movable on $X$ for every $i$. 
Thus $K_X+\Delta$ is a limit of movable $\mathbb R$-divisors in 
$N^1(X/Y)$. So $K_X+\Delta\in \overline{\Mov} (X/Y)$. 
It is a contradiction. 
Therefore the sequence of flips terminates after finitely 
many steps.
\end{proof}

\section{Proofs}\label{sec2}

In this section, we will prove various results stated in Section \ref{sec1} 
as applications of 
Theorem \ref{33}. 

\begin{proof}[Proof of {\em{Theorem \ref{main}}}] 
Before we run the minimal model program with scaling, 
we note the following easy observation. 

\begin{step}[{cf.~\cite[Proposition 4.2]{fm}}]\label{step1}
Let $m$ be a positive 
integer such that 
$m(K_X+\Delta)$ is Cartier and $m(K_X+\Delta)|_F\sim 0$ where 
$F$ is the generic fiber of $f$. 
Then we have a natural injection 
$$
0\to f^*f_*\mathcal O_X(m(K_X+\Delta))\to \mathcal O_X(m(K_X+\Delta))
$$
because $f_*\mathcal O_X(m(K_X+\Delta))$ is torsion-free and 
$Y$ is a smooth curve. 
Therefore, 
there is a $\mathbb Q$-divisor 
$D$ on $Y$ and an effective 
$\mathbb Q$-divisor $B$ on $X$ such that 
$B$ is vertical with 
respect to $f$, 
$$
K_X+\Delta\sim _{\mathbb Q}f^*D+B, 
$$ 
and $\Supp B$ does not contain any fibers of $f$. 
We note that $K_X+\Delta$ is $f$-nef if and 
only if $B=0$, equivalently, 
$K_X+\Delta\sim _{\mathbb Q, Y}0$ 
(cf.~\cite[III.~(8.2) Lemma]{BPV}). 
\end{step}

\begin{step}
We take an effective $\mathbb Q$-Cartier 
$\mathbb Q$-divisor $H$ on $X$ such that 
$H$ is big, $(X, \Delta+H)$ is 
dlt, $K_X+\Delta+H$ is nef over $Y$, 
and $\mathbf B_+(H/Y)$ contains no lc 
centers of $(X, \Delta)$. 
We run the $(K_X+\Delta)$-minimal model program with scaling of 
$H$ over $Y$ as in 
\ref{32}. 
Since divisorial contractions can occur only finitely many times, 
we can assume that every step is a flip. 
Since $B\not\sim_{\mathbb Q, Y}0$, 
we can find an irreducible component $E$ of 
$\Supp B$ such that 
$$
B\cdot A^{n-2}\cdot E<0.
$$ 
where $n=\dim X$ and $A$ is an $f$-ample 
Cartier divisor on $X$. 
This is essentially Zariski's lemma (cf.~\cite[III.~(8.2) Lemma]{BPV}). Thus 
$$
(K_X+\Delta)\cdot A^{n-2}\cdot E<0.  
$$ 
Assume that $K_X+\Delta\in \overline {\Mov}(X/Y)$. 
Then 
$$
(K_X+\Delta)\cdot A^{n-2}\cdot E\geq 0. 
$$
Therefore, $K_X+\Delta \not \in \overline {\Mov} (X/Y)$. 
Thus the $(K_X+\Delta)$-minimal model program terminates 
by Theorem \ref{33}. 
\end{step}
\begin{step} 
On the output $X_m$ of the 
minimal model program, 
$K_{X_m}+\Delta_m\sim _{\mathbb Q, Y}B_m$ where 
$B_m$ is the pushforward 
of $B$ on $X_m$. Since 
$B_m$ is nef over $Y$, 
$B_m\sim _{\mathbb Q, Y}0$ (cf.~\cite[III.~(8.2) Lemma]{BPV}). 
Therefore, $K_{X_m}+\Delta_m\sim _{\mathbb Q, Y}0$. 
\end{step}
We complete the proof of Theorem \ref{main}. 
\end{proof}

\begin{rem}
Let $f:(X, \Delta)\to Y$ be a projective 
dlt morphism from a $\mathbb Q$-factorial 
dlt pair $(X, \Delta)$ (cf.~\cite[Definition 7.1]{kollar-mori}). 
Assume that $K_X+\Delta$ is $f$-nef 
over a non-empty Zariski open set $U\subset Y$. 
Then the special termination (see, 
for example, \cite[Theorem 4.2.1]{special}) implies 
that any sequence of flips 
in the $(K_X+\Delta)$-minimal 
model 
program over $Y$ must terminate. 
We note that the special termination 
has been proved only in dimension $\leq 4$ (see, 
for example, \cite[Theorem 4.2.1]{special}). 
\end{rem}

Let us prove Theorems \ref{13}, \ref{14}, \ref{15}, and Corollary \ref{12}. 

\begin{proof}[Proof of 
{\em{Theorem \ref{13}}}] 
By the assumptions, 
$f:X\to Y$ is a dlt morphism (cf.~\cite[Definition 7.1]{kollar-mori}). 
By applying Theorem \ref{main}, 
we obtain a relative 
minimal model 
$f_m:X_m\to Y$ of $f:X\to Y$. 
We see that $f_m:X_m\to Y$ is automatically a 
dlt morphism. 
We note that $X_m$ is $\mathbb Q$-factorial and 
has only terminal singularities. 
By adjunction, 
$$
(K_{X_m}+S)|_S=K_S
$$ 
and $S$ is semi divisorial log terminal 
because $(X_m, S)$ is dlt (cf.~\cite[Remark 1.2 (3)]{fujino-semi}). 
By considering the following natural 
injection 
$$
0\to f^*f_*\mathcal O_{X_m}(K_{X_m})\to \mathcal O_{X_m}(K_{X_m}), 
$$ 
which is also surjective outside the special 
fiber $S$, 
as in Step \ref{step1} in the proof of Theorem \ref{main}, we 
obtain $K_{X_m}\sim 0$ because 
$K_{X_m}$ is nef over $Y$. 
In particular, $K_S\sim 0$ by adjunction. 
\end{proof}

\begin{proof}[Proof of {\em{Theorem \ref{14}}}]
The proof of Theorem \ref{13} works in this setting. 
If $S$ is reducible, semi divisorial log terminal, and 
$K_S\sim _\mathbb Q 0$, 
then we will show that every irreducible 
component of $S$ is uniruled. 
Let $S_0$ be an irreducible 
component of $S$. Then 
$K_{S_0}+\Theta\sim _{\mathbb Q}0$ with 
an effective $\mathbb Q$-divisor 
$\Theta \ne 0$ because $S$ is connected. 
Therefore, $S_0$ is uniruled by \cite[Corollary 2]{mm}. 
From now on, we assume that 
$S$ is irreducible. 
If $S$ has only canonical 
singularities, then $S$ is not uniruled 
because $K_S\sim _{\mathbb Q}0$. 
If $S$ is not canonical, 
then we take a dlt blow-up (cf.~Theorem \ref{41}) and 
obtain a birational morphism 
$\varphi:T\to S$ from a normal projective variety 
$T$ such that $K_T=\varphi^*K_S-E$ 
where $E$ is effective and $E\ne 0$. 
Therefore, $K_T\sim _\mathbb Q-E\ne 0$. 
Thus $T$ is uniruled 
by \cite[Corollary 2]{mm}. 
It implies that $S$ is uniruled. 
\end{proof}

\begin{proof}[Proof of {\em{Theorem \ref{15}}}]
The former part follows from Theorem \ref{main}. 
We will check the latter part. 
We assume that $S$ is reducible. 
Let $E$ be any irreducible component of $S$, 
and let $\varepsilon$ be a sufficiently 
small positive rational number. 
Apply Theorem \ref{main} to $(X, \varepsilon E)$ over $Y$. 
Then it is easy to see that 
the divisor $E$ must be contracted in this minimal 
model program. 
Therefore $E$ is uniruled by \cite[Proposition 5-1-8]{kmm}. 
From now on, we assume 
that $S$ is irreducible. 
It is sufficient to see that 
$S$ is uniruled when $S$ is not canonical. 
First we assume that $S$ is normal. 
Then we take a dlt blow-up $\varphi:T\to S$ 
(cf.~Theorem \ref{41}). 
We can write $K_T=\varphi^*K_S-E$ such that 
$E\ne 0$ is effective. 
Therefore, $T$ is uniruled by 
\cite[Corollary 2]{mm} because $K_T\sim _{\mathbb Q}-E\ne 0$. 
Thus $S$ is uniruled. 
Next we assume that $S$ is not normal. 
Let $\nu:S^\nu\to S$ be the normalization. 
Then 
$$
K_{S^\nu}+\Theta=\nu^*K_S\sim _\mathbb Q 0
$$ 
such that 
$\Theta$ is effective 
and $\Theta\ne 0$. 
We note that $S$ is Cohen--Macaulay since 
$X$ is Cohen--Macaulay and $S$ is $\mathbb Q$-Cartier 
(cf.~\cite[Corollary 5.25]{kollar-mori}). 
Therefore, 
$S^\nu$ is uniruled by \cite[Corollary 2]{mm}. 
Thus $S$ is uniruled. 
Anyway, $S$ is not uniruled if and only if $S$ has only 
canonical singularities.  
\end{proof}

\begin{proof}[Proof of {\em{Corollary \ref{12}}}] 
Let $H$ be a general effective $f$-big 
divisor on $X$ such that 
$K_X+H$ is $f$-nef and 
$(X, H)$ is dlt. 
We run the minimal model program with scaling of $H$ over $Y$. 
Then, by 
\cite[Proposition 2.7]{lai}, we can 
assume that the general fiber of $f:X\to Y$ is 
a good minimal model. 
By Theorem \ref{main}, 
this minimal model program terminates after finitely many steps. 
\end{proof}

\section{Dlt blow-ups}\label{sec-dlt}
In this section, we will give 
a slight refinement of 
\cite[17.10 Theorem]{FA} and \cite[Corollary 1.4.3]{bchm} as 
an application of Theorem \ref{33}. 
See also \cite[\S 10]{fujino}.  

\begin{thm}[Dlt blow-ups]\label{41}
Let $X$ be a normal quasi-projective 
variety and 
let $\Delta$ be an $\mathbb R$-boundary 
divisor on $X$ such that $K_X+\Delta$ is $\mathbb R$-Cartier. 
Let $f:W\to X$ be a resolution such that 
$\Exc(f)\cup \Supp f^{-1}_*\Delta$ is a simple normal crossing 
divisor 
on $W$ where $\Exc (f)$ is the exceptional 
locus of $f$. 
Let $\mathcal E$ be a subset of the $f$-exceptional 
divisors $\{E_i\}$ with the following properties{\em{:}}
\begin{itemize}
\item If $a(E_i, X, \Delta)\leq -1$, then 
$E_i\in \mathcal E$. 
\item If $E_i\in \mathcal E$, then $a(E_i, X, \Delta)\leq 0$.  
\end{itemize}
Then there is a factorization
$$f:W\overset{h}{\dashrightarrow} Z\overset {g}{\longrightarrow} X$$
with the following properties{\em{:}}
\begin{itemize}
\item[(a)] $h$ is a local isomorphism at every 
generic point of $E_i\in \mathcal E$, 
\item[(b)] $h$ contracts every exceptional divisor 
not in $\mathcal E$, 
\item[(c)] 
we have 
\begin{align*}
&h_*(K_W+f^{-1}_*\Delta +\sum _{a_i\geq -1}-a_iE_i+
\sum _{a_i<-1}E_i)\\&
=K_Z+g^{-1}_*\Delta +\sum _{E_i\in \mathcal E, \ a_i\geq -1}-a_ih_*E_i+
\sum _{a_i<-1}h_*E_i\\
&=g^*(K_X+\Delta)+\sum _{a_i<-1}(a_i+1)h_*E_i, 
\end{align*}
where $a_i=a(E_i, X, \Delta)$, and 
\item[(d)] the pair 
$$
(Z, g^{-1}_*\Delta+\sum _{E_i\in \mathcal E, \ 
a_i \geq -1}-a_i h_*E_i+\sum _{a_i<-1}h_*E_i)
$$ 
is a $\mathbb Q$-factorial dlt pair. 
\end{itemize}
In particular, if $(X, \Delta)$ is log canonical, 
then $$(Z, g^{-1}_*\Delta+\sum _{E_i\in \mathcal E, a_i\geq -1}-a_ih_*E_i)$$ 
is dlt and 
$$
K_Z+g^{-1}_*\Delta+\sum _{E_i\in \mathcal E, a_i \geq -1} 
-a_ih_*E_i=g^*(K_X+\Delta). 
$$
\end{thm} 
\begin{proof}
For a small $\varepsilon>0$, we put  
$$
d(E_i)=\begin{cases}
1&{\text{if}} \ \ a(E_i, X, \Delta)< -1\\
-a(E_i, X, \Delta) & \text{if}\ \ E_i\in \mathcal E, a(E_i, X, \Delta)\geq -1\\
\max \{-a(E_i, X, \Delta)+\varepsilon, 0\} & \text{if} \ \ E_i\not \in \mathcal E. 
\end{cases}
$$ 
We take a general effective Cartier divisor 
$H$ on $Z$ such that 
$(W, f^{-1}_*\Delta+\sum d(E_i)E_i+H)$ is dlt and 
that $K_W+f^{-1}_*\Delta +\sum d(E_i)E_i+H$ is $f$-nef. 
We run the 
$
(K_W+f^{-1}_*\Delta+\sum d(E_i)E_i)$-minimal model program 
with scaling of $H$ over $X$. 
We note that 
\begin{align*}
&K_W+f^{-1}_*\Delta+\sum d(E_i)E_i\\&=
f^*(K_X+\Delta)+\sum _{E_i\not\in \mathcal E}(d(E_i)+a_i)E_i+\sum 
_{a_i<-1}(1+a_i)E_i. 
\end{align*}
Since divisorial contractions can occur finitely many times, 
we can assume that 
every step of the minimal model program is a flip. 
We put 
$$
E=\sum _{E_i\not\in \mathcal E}(d(E_i)+a_i)E_i+\sum 
_{a_i<-1}(1+a_i)E_i.
$$
Then $E$ is exceptional over $X$. 
We assume that $\sum _{E_i\not \in \mathcal E}(d(E_i)+a_i)E_i\ne 0$. 
Then $E\not\in \overline {\Mov}(W/X)$ by Lemma \ref{lem42} below. Therefore, 
any sequence of flips terminates after finitely many steps by Theorem \ref{33}. 
However, $E$ can not become nef over $X$ by flips since $-E$ is not 
effective. It is a contradiction. Therefore, 
$\sum _{E_i\not \in \mathcal E}(d(E_i)+a_i)E_i=0$. 
It completes the proof.
\end{proof}

The lemma below is a variant of the well-known negativity lemma. 

\begin{lem}\label{lem42}
Let $f:X\to Y$ be a birational morphism 
from a normal $\mathbb Q$-factorial algebraic 
variety $X$. 
Let $E$ be an $\mathbb R$-divisor on $X$ such that 
$\Supp E$ is $f$-exceptional and $E \in \overline {\Mov}(X/Y)$. 
Then $-E$ is effective. 
\end{lem}
\begin{proof}
We write $E=E_+-E_-$ such that 
$E_+$ and $E_-$ have no common irreducible components and that 
$E_+\geq 0$ and $E_-\geq 0$. 
We assume that 
$E_+\ne 0$. 
Let $A$ (resp.~$H$) be 
an ample Cartier divisor on $Y$ (resp.~$X$). 
Then we can find an irreducible component $E_0$ of $E_+$ such that 
$$
E_0\cdot (f^*A)^k\cdot H^{n-k-2}\cdot E<0
$$ 
where $\dim X=n$ and $\codim _Yf(E_+)=k$. 
On the other hand, 
$$
E_0\cdot (f^*A)^k\cdot H^{n-k-2}\cdot E\geq 0 
$$ 
if $E\in \overline {\Mov}(X/Y)$.  It is 
a contradiction. 
Therefore, $-E$ is effective. 
\end{proof}

\end{document}